\documentclass[12pt]{amsart}
\usepackage{amsthm}
\usepackage{amssymb}
\usepackage[rgb]{xcolor}
\usepackage[margin=1.0in]{geometry}
\usepackage{hyperref,color}
\newtheorem{thrm}{Theorem}

\newtheorem{prop}[thrm]{Proposition}
\newtheorem{lemma}[thrm]{Lemma}
\newtheorem{cor}[thrm]{Corollary}
\newtheorem{conj}[thrm]{Conjecture}

\newcommand{\cee}{\mathbb{C}}
\newcommand{\arr}{\mathbb{R}}

\newcommand{\ii }{\sqrt{-1}}
\newcommand{\N}{\nabla}

\newcommand{\ag}{\alpha}
\newcommand{\bg}{\beta}
\newcommand{\dg}{\delta}
\newcommand{\ga}{\gamma}
\newcommand{\ab}{{\alpha\beta}}
\newcommand{\delbar}{\bar{\partial}}
\newcommand{\del}{\partial}

\numberwithin{thrm}{section}
\begin{document}
\title{$\delbar$-harmonic maps between almost Hermitian manifolds}
\author{Jess Boling}
\email{\href{mailto:jboling@uci.edu}{jboling@uci.edu}}
\address{Rowland Hall\\
         University of California\\
         Irvine, CA 92617}
\begin{abstract}
In this paper we study an energy of maps between almost Hermitian manifolds for which pseudo-holomorphic maps are global minimizers. We derive its Euler-Lagrange equation, the $\delbar$-harmonic map equation, and show that it coincides with the harmonic map equation up to first order terms. We prove results analogous to the those that hold for harmonic maps, including obstructions to the long time existence of the associated parabolic flow, an Eells-Sampson type result under appropriate conditions on the target manifold, and a bubbling result for finite time singularities on surfaces. We also consider examples of the flow where the target is a non-K\"ahler surface.
\end{abstract}
\maketitle

\section{Introduction and Statement of Main Results}
Given two Riemannian manifolds $(M,g)$ and $(N,h)$, a map $f:M\to N$ is said to be harmonic if it is a critical point of the energy functional
\begin{align*}
E(f)=\frac{1}{2}\int_M|Tf|^2dV
\end{align*}
with respect to compactly supported test variations of $f$, where $|Tf|^2$ is the norm squared of the derivative $Tf:TM\to TN$ as a section of $T^*M\otimes f^{-1}TN$. The Euler-Lagrange equation associated to $E$ is called the tension of $f$ and has the local coordinate expression
\begin{align*}
(tr_g\N Tf) ^i=\tau^i(f)=g^{\ab}(f^i_{\ab}-f^i_{\ga}\Gamma^{\ga}_{\ab}+f^j_\ag f^k_\bg\Gamma^i_{jk}).
\end{align*}
Here we use the convention that coordinates on $M$ are denoted with the Greek indices $\alpha$, $\beta$, $\ldots$ and so on while Roman indices $i$, $j$, $\ldots$  will denote coordinates on the target $N$. Harmonic maps are well studied objects in geometric analysis, with many results obtained toward their existence, regularity, and compactness; see the surveys in \cite{LW}, \cite{SY}, and \cite{EL}. We remind the reader specifically of the celebrated existence result of Eells and Sampson.
\begin{thrm}
Suppose that $(M,g)$ and $(N,h)$ are closed Riemannian manifolds where the sectional curvatures of $h$ are non-positive. Then any smooth $f_0:M\to N$ is homotopic to a harmonic map.
\end{thrm}

This result, proved with gradient descent methods, has inspired many of the heat flows which are topics of current research in differential geometry. For example Hamilton \cite{HAM} cites the Eells-Sampson result as the motivator for studying Ricci flow. In the current paper, inspired by the result of Eells-Sampson, we study a functional of maps between almost Hermitian manifolds and its associated parabolic flow.

Let $f:M\to N$ be a differentiable map between the almost Hermitian manifolds $(M,J_M,g)$ and $(N,J_N,h)$, and let its derivative be $Tf:TM\to TN$. In the case where the complex structures are integrable, $f$ is holomorphic in local complex coordinates if, and only if, $J_N Tf=Tf J_M$. In situations where complex coordinates do not exist, a map satisfying $J_N Tf=Tf J_M$ is said to be complex or pseudo-holomorphic. Note that the derivative $Tf:TM\to TN$ has an orthogonal decomposition
\[Tf=\frac{1}{2}(Tf+J_NTfJ_M)+\frac{1}{2}(Tf-J_NTfJ_M)\]
where the first component of this decomposition vanishes if, and only if, $f$ is pseudo-holomorphic.

Holomorphic maps are important objects in complex and symplectic geometry, as their moduli often contain information on the global structure of the source and target manifolds. For example, Gromov invariants count families of holomorphic maps from curves into a fixed symplectic manifold \cite{GW}, while the classification of compact complex surfaces with $b_1=1$ would be complete if certain homology classes had holomorphic representatives \cite{NAK}. 

Counting such curves motivates the search for holomorphic maps in a given homotopy class, and a perhaps naive attempt towards finding these maps is to apply gradient descent methods to functionals which are small when evaluated at holomorphic maps. For this reason, it is natural to define the pseudo-holomorphic energy of a map $f:(M,g,J_M)\to (N,h,J_N)$ to be
\[E_+(f)=\frac{1}{4}\int_M|Tf+J_NTfJ_M|^2dV_g.\]
This energy was first studied by Lichnerowicz in \cite{LN}, where it is called $E''(f)$. In that paper the relationship between $E_+$ and the harmonic map energy $E$ is studied for maps between almost K\"ahler manifolds.

Our first result is to give explicit formulas for the first and second variations of $E_+$, exact expressions of which are not found in the literature.
\begin{prop}
Let $f:M\times(-\epsilon,\epsilon)\to N$ be a smooth one parameter family of maps between the almost Hermitian manifolds $(M,g,J_M)$ and $(N,h,J_N)$ with compactly supported variation field $\partial_tf=v$. Then
\begin{align*}
\frac{d}{dt}E_+(f)=-\int_M\langle \tau+A,v\rangle dV,
\end{align*}
where $\tau$ is the tension of $f$ and $2A=\langle d^*\omega_M,f^\times\omega_N\rangle+\langle \omega_M,f^\times d\omega_N\rangle$, which has the local coordinate expression
\begin{align*}
2A^i=(d^*\omega)_{\ag}f^j_\bg\omega_{lj}g^{\ab}h^{li}+\omega_{\ab}f^j_\ga f^k_\delta (d\omega)_{ljk}g^{\ag\ga}g^{\bg\delta}h^{li},
\end{align*}
where $\omega_M(X,Y)=g(JX,Y)$.
\end{prop}
Maps $f$ satisfying $\tau_+(f)=\tau(f)+A(f)=0$ will be called $\delbar$-harmonic. We pause to make some small observations about this expression. First, one immediately gets the result of Lichnerowicz in \cite{LN} on the equivalence between $E_+$ and $E$ critical maps if the source is balanced and the target is almost K\"ahler; we will give explicit, non-K\"ahler examples which show that $E_+$ critical maps are distinct from $E$ critical maps in general. The second observation is that, independent of any integrability assumptions on source or target, the Euler-Lagrange equation of $E_+$ is elliptic and semi-linear. We then compute the second variation of $E_+$ at a $C^2$ critical point.
\begin{prop}
Let $f:M\to N$ be a $C^2$ critical point of $E_+$ with respect to all compactly supported variations $\partial_tf=v$, so that $\tau_+(f)=\tau(f)+A(f)=0$. Then the second variation of $E_+$ at $f$ is given by
\begin{align*}
\frac{d^2}{dt^2}|_{t=0}E_+(f)=\int_M\langle Lv,v\rangle dV,
\end{align*}
where $L$ is the operator on sections of $f^{-1}TN$ given by
\begin{align*}
-Lv = \Delta v+tr_gR(& v,Tf)Tf +\frac{1}{2}\{ \langle d^*\omega_M,\omega_N(\cdot,\N v)\rangle^\sharp\\& +\langle d^*\omega_M,(\N_v\omega_N)(\cdot,Tf)\rangle^\sharp+\langle\omega_M,(d\omega_N)(\cdot,\N v,Tf)\rangle^\sharp\\& +\langle\omega_M,(d\omega_N)(\cdot,Tf,\N v)\rangle^\sharp\\& +\langle\omega_M,(\N_v d\omega_N)(\cdot,Tf,Tf)\rangle^\sharp\}.
\end{align*}
\end{prop}
A simple corollary of the previous proposition is:
\begin{cor}
Suppose that $M$ is closed and $f:M\to N$ is pseudo-holomorphic. Then the operator $L:C^\infty(f^{-1}TN)\to C^\infty(f^{-1}TN)$ in the previous proposition is non-negative. In particular, if $(M,g,J)$ is a closed almost Hermitian manifold the operator $L:C^\infty(TM)\to C^\infty(TM)$, given by
\begin{align*}
-Lv = \Delta v+Rc(& v)+\frac{1}{2}\{\langle d^*\omega,\omega(\cdot,\N v)+(\N_v\omega)(\cdot,id)\rangle^\sharp\\& +\langle\omega,d\omega(\cdot,\N v,id)\rangle^\sharp+\langle\omega,d\omega(\cdot,id,\N v)\rangle^\sharp\\&+\langle\omega,(\N_vd\omega)(\cdot,id,id)\rangle^\sharp\},
\end{align*}
is non-negative. Moreover, if $v\in C^\infty(TM)$ is a pseudo-holomorphic vector field then $Lv=0$.
\end{cor}
The non-negativity of $L$ and its relationship to holomorphic vector fields is known in the K\"ahler setting. For example, by applying the non-negativity of $L$ to the gradient of an eigenfunction of the Laplacian, a version of the Obata-Lichnerowicz theorem for K\"ahler manifolds is proved in \cite{URA}. We can then give a similar lowest eigenvalue bound in the almost-K\"ahler setting.
\begin{cor}
If $(M,g,J)$ is a compact, almost K\"ahler manifold with Ricci curvature $Rc\geq\alpha>0$ then the first eigenvalue of the Laplacian satsfies $\lambda_1\geq 2\alpha$.
\end{cor}

Our next proposition concerns the well posedness of the gradient flow associated to $E_+$ and shows that it has the same obstructions to long time existence as the harmonic map heat flow.
\begin{prop}
Let $f_0:M\to N$ be a smooth map between the closed almost Hermitian manifolds $(M,g,J_M)$ and $(N,h,J_N)$. Then there is a maximal $T$ such that the initial value problem
\begin{align*}
\partial_tf&=\tau_+(f)\\
f|_{t=0}&=f_0
\end{align*}
has a unique smooth solution on $M\times[0,T)$. Moreover, if $T<\infty$ then
\begin{align*}
\lim_{t\to T}|Tf|_{C^0}=\infty.
\end{align*} 
\end{prop}
Equipped with this proposition, we prove the following long time existence theorem for solutions to the $\delbar$-harmonic map heat flow with negatively curved, almost K\"ahler targets.
\begin{thrm}
Let $f_0:M\to N$ be a smooth map between the closed almost Hermitian manifolds $(M,g,J_M)$ and $(N,h,J_N)$. Suppose that $\omega_h$ is almost K\"ahler and the sectional curvature of $h$ is negative. Then the solution to the $\delbar$-harmonic map heat flow with initial condition $f_0$ has a unique smooth solution on $M\times[0,\infty)$.
\end{thrm}
In some cases, for example when $J_N$ is integrable, the negative curvature assumption in this result can be weakened to non-positive curvature. We will also discuss the difficulties in improving this result to the general case where $d\omega_N\neq 0$; stronger conditions on the curvature in relation to the complex structure are needed for our proof to go through without substantial modification.

Convergence of the harmonic map heat flow at infinite time requires a parabolic Harnack inequality to prove a bound on $|Tf|^2$ on all of $M\times [0,\infty)$ given a uniform bound on $|Tf|_{L^2}$. In the case of the harmonic map energy this bound on $|Tf|_{L^2}=2E(f)$ is free because the flow is precisely given by following the negative $L^2$-gradient of $E$. In the context of the previous theorem, we do not have such a bound because the pseudo-holomorphic energy $E_+$ is in general non-coercive; we do not expect the energy to be bounded at infinite time. In the presence of such a bound, convergence to a $\delbar$-harmonic map at infinite time follows. We attempt to get around this difficulty by considering a family of energies $E_a$ which are coercive if $|a|<1$ and which contain $E_+=E_1$ as a limiting case. This family is given by linear interpolation $E_a=(1-a)E+aE_+$ between the harmonic map energy $E_0=E$ and the pseudo-holomorphic energy $E_1=E_+$. Because $(1-|a|)E\leq E_a$, we obtain the following existence result for these modified functionals.
\begin{thrm}\label{aconv}
Let $f:M\to N$ be a smooth map between the closed, almost Hermitian manifolds $(M,g,J_M)$ and $(N,h,J_N)$. Suppose that the sectional curvatures of $N$ are negative and that $\omega_h$ is almost K\"ahler. Then for all $|a|<1$ the parabolic flow corresponding to the functional $E_a$, beginning with $f$, exists for all time and has subsequential convergence to a critical point of $E_a$.
\end{thrm}
The question is then whether or not this result can be used as a stepping stone toward the existence of $E_+$ critical maps. One possible way of getting around the non-coerciveness of $E_+$ is to take a sequence $f_i$ of $E_{a_i}$ critical maps with $a_i\uparrow 1$ and extract appropriate subsequences which converge to $\delbar$-harmonic maps, but we leave this question to future work.

An additional case of interest is when $M=\Sigma$ is a compact Riemann surface and $N$ is an arbitrary compact almost Hermitian manifold, not necessarily almost K\"ahler. This is the critical dimension for the functionals we consider, and just like the harmonic map energy, $E_+$ is conformally invariant in this case. Understanding $\delbar$-harmonic maps in this case is particularly important given the use of $J$-holomorphic curves in almost Hermitian geometry. We consider examples of the flow in this case as well as prove the existence of $\delbar$-harmonic bubbles at a finite time singularity. 

Specifically, in Section 4 we work through an interesting example of the flow restricted to a family of harmonic tori $f:T^2\to S^3\times S^1$ inside a Hopf surface. This family is parameterized by orthonormal pairs in $\arr^4$ and we show that the flow both preserves this family, restricting to an ODE, and converges to a holomorphic or anti-holomorphic map at infinite time. 

Under the assumption of a uniform energy bound we also prove the following bubbling theorem for finite time singularities to the $\delbar$-harmonic map heat flow.
\begin{thrm}
Let $f_0:\Sigma\to N$ be a smooth map of a compact Riemann surface into a compact, almost Hermitian manifold $N$. Suppose that the $\delbar$-harmonic map heat flow beginning with $f_0$ exists on a maximal time interval $[0,T)$ where $T<\infty$ and there is a uniform energy bound $E(f)<C$ along a solution to the flow. Then there is a non-constant $\delbar$-harmonic map $\theta:S^2\to N$.
\end{thrm}
\subsection*{Acknowledgments.}
The author would like to thank Jeff Streets for many helpful conversations.
\section{Background and Variation Computations}
In this section we will give a brief overview of some calculations useful in the context of harmonic maps, establish the notation which is used throughout the paper, and derive the first and second variations of the anti-holomorphic energy $E^+$.

In what follows $(M,g,J_M)$ and $(N,h,J_N)$ are almost Hermitian manifolds, assumed complete and without boundary. When writing some object in local coordinates we reserve Greek indices $\alpha,\beta,\ldots$ for coordinates on $M$ and Roman indices $i,j,\ldots$ for coordinates on $N$. We will also use the standard summation convention unless stated otherwise, and will often abbreviate coordinate derivatives of functions with subscripts: $\partial_\ag u=u_\ag$. 

Let $f:M\times I\to N$ be a smooth one-parameter family of maps defined on some open interval $I$. For each $t\in I$, let $Tf:M\to N$ be the derivative of the map $f(\cdot,t)$, viewed as a section of $T^*M\otimes f(\cdot,t)^{-1}TN$, with $f^{-1}E$ denoting the pullback bundle. In local coordinates we have $Tf=f^i_\ag \partial_i\otimes d^\ag$. Note that $Tf$ can also be viewed as a section of $T^*(M\times I)\otimes f^{-1}TN$ by pre-composing the (full) derivative of the family $f_*:T(M\times I)\to TN$ with the canonical projection to $TM$.

The manifolds $M$, $M\times I$, and $N$ have Levi-Civita connections which induce connections on the various tensor and pullback bundles that we will consider. The symbol $\N$ will denote the full covariant derivative operator with respect to spacial variables only, meaning if $s$ is a section of some bundle over $M\times I$ we pre-compose the full covariant derivative of $s$ with the projection to $TM$, so that $(\N s)^i_\ag=s^i_\alpha+s^jf^k_\ag\Gamma_{kj}^i$ in coordinates. We reserve $\N_t$ for the covariant derivative in the direction of $\partial_t$. Let $v=f^i_t\partial_i$ denote the variational vector field of $f$, which in various notations can be written
\begin{align*}
v=\frac{\partial f}{\partial t}=f_*(\partial_t).
\end{align*}
By unwinding the notation we have:
\begin{lemma}
Let $f$ be a smooth one-parameter family of maps with variational vector field $v$. Then
\begin{align*}
\N_tTf=\N v.
\end{align*}
\end{lemma}
We next recall the derivation of the tension tensor $\tau$ as the Euler-Lagrange equation of the energy functional. Given a map $f:M\to N$, let its energy density be $e(f)=\frac{1}{2}|Tf|^2$, where in coordinates
\begin{align*}
|Tf|^2=g^{\ag\bg}f^i_\ag f^j_\bg h_{ij}.
\end{align*}
The energy of $f$ is then the integral $E(f)=\int_M e(f)dV_g$ of the energy density.
\begin{prop}
Let $f$ be a smooth one parameter family of maps whose variational vector field $v$ has compact support. Then
\begin{align*}
\frac{d}{dt}E(f)=\int_M\langle \N v,Tf\rangle dV_g=-\int_M\langle v,\tau\rangle dV_g,
\end{align*}
where $\tau=tr_g\N Tf$ is the tension tensor of $f$, given in local coordinates by
\begin{align*}
\tau^i=g^{\ag\bg}(f^i_{\ag\bg}-f^i_\ga\Gamma^\ga_{\ag\bg}+f^j_\ag f^k_\bg \Gamma^i_{jk}).
\end{align*}
\begin{proof}
The first equality is immediate from the previous lemma by differentiating under the integral, while the second follows from the divergence theorem applied to the vector field $X=\langle v,Tf\rangle^\sharp$.
\end{proof}
\end{prop}
The next proposition concerns the second variation of the energy.
\begin{prop}
Let $f$ be a smooth one parameter family of maps whose variational vector field $v$ has compact support. Then
\begin{align*}
\frac{d^2}{dt^2}E(f)=-\int_M\langle \N_t v,\tau\rangle+\langle v, \Delta v + tr_g R^N(v,Tf)Tf\rangle dV_g,
\end{align*}
where $\Delta v=tr_g\N^2 v$.
\begin{proof}
The variation of $\tau$ is given by
\begin{align*}
\N_t\tau=\N_t tr_g\N Tf&=tr_g\N_t\N Tf\\
&=tr_g\N\N_t Tf+tr_gR^N(v,Tf)Tf\\
&=tr_g\N^2v+tr_gR^N(v,Tf)Tf,
\end{align*}
where $R^N(X,Y)=\N_X\N_Y-\N_Y\N_X-\N_{[X,Y]}$ is the curvature tensor of $N$. The result follows from differentiation under the integral of the previous proposition, noting that $\partial_t\langle v,\tau\rangle=\langle \N_t v,\tau\rangle+\langle v,\N_t\tau\rangle$.
\end{proof}
\end{prop}
Next we look at what can be done when the (almost) complex structures are taken into account. We first focus on the orthogonal decomposition of $Tf$ and the norm of its various pieces.
\begin{lemma}
Let $f:M\to N$ be a differentiable map between the almost Hermitian manifolds $(M,g,J_M)$ and $(N,h,J_N)$. Then the derivative has an orthogonal decomposition
\begin{align*}
Tf=\frac{1}{2}(Tf+J_NTfJ_M)+\frac{1}{2}(Tf-J_NTfJ_M)
\end{align*}
\begin{proof}
We compute
\begin{align*}
\langle Tf+J_NTfJ_M,Tf-J_NTfJ_M\rangle=|Tf|^2-|J_NTfJ_M|^2=|Tf|^2-|Tf|^2=0.
\end{align*}
\end{proof}
\end{lemma}
\begin{lemma}
With the same assumptions of the previous lemma, we have
\begin{align*}
\frac{1}{4}|Tf+J_NTfJ_M|^2=\frac{1}{2}|Tf|^2+\frac{1}{2}\langle Tf,J_NTfJ_M\rangle
\end{align*}
and
\begin{align*}
\langle Tf,J_NTfJ_M\rangle=-\langle\omega_M,f^*\omega_N\rangle,
\end{align*}
where the inner product on two-forms has the normalization $\langle a,b\rangle=a_{\ag\bg}b_{\ga\dg}g^{\ag\ga}g^{\bg\dg}$, i.e. is the inner product as real tensors.
In particular, if $M$ is a Riemann surface we have
\begin{align*}
-\frac{1}{2}\langle\omega_M,f^*\omega_N\rangle dV_g=-f^*\omega_N
\end{align*}
\begin{proof}
The first equality comes from expanding the inner product, while the second is done in coordinates:
\begin{align*}
\langle Tf,J_NTfJ_M\rangle=f^i_\ag J^k_jf^j_\bg J^\bg_\ga g^{\ag\ga}h_{ik}&=-J^\bg_\ga g^{\ag\ga}(f^*\omega_N)_{\ag\bg}\\&=-(\omega_M)_{\ga\dg}g^{\ga\ag}g^{\dg\bg}(f^*\omega_N)_{\ag\bg}.
\end{align*}
The proof is completed by noting that on a Riemann surface $|\omega_M|^2=2$ and $dV_g=\omega_M$, so that $\frac{1}{\sqrt{2}}\omega_M$ is an orthonormal basis of $\Omega^2$ at each point and $\langle\omega_M,f^*\omega_N\rangle dV_g=2f^*\omega_N$.
\end{proof}
\end{lemma}
The remark about the normalization of the inner product of forms is necessary because it is not the convention used in much of the literature. For example, if $M$ is $2m$ dimensional we have $|\omega_M|^2=2m$, while other authors would say $|\omega_M|^2=m$. We note that the volume form is still given by $dV_g=\frac{1}{m!}\omega_M^m$ and that our normalization does not change the adjoint $d^*$ of the exterior derivative.

Our next proposition concerns how the pullback $f^*\omega$ of a two-form varies with a variation of $f$ and is necessary for computing the first variation of $E^+$. If $f$ is a one parameter group of diffeomorphisms of $M$, the following is nothing more than the familiar Cartan formula $\mathcal{L}_X\omega=d\iota_X\omega+\iota_Xd\omega$, but for arbitrary smooth families of maps $f:M\to N$ the generalization we give is perhaps unknown to the reader.
\begin{prop}
Let $f:M\times I\to N$ be a smooth one parameter family of maps with $\partial_tf=v$. Fix a differential two-form $\omega$ on $N$, and let $f^*\omega$ be the pullback of $\omega$ with respect to the map $f(t):M\to N$. Then as a one parameter family of forms on $M$ we have
\begin{align*}
\frac{d}{dt}f^*\omega=df^*\iota_v\omega+f^*\iota_vd\omega.
\end{align*}
\end{prop}
\begin{proof}
We compute directly in coordinates
\begin{align*}
(\frac{d}{d t}f^*\omega)_{\alpha\beta} = v^i_\alpha f^j_\beta\omega_{ij}+f^i_\alpha v^j_\beta\omega_{ij}+f^i_\alpha f^j_\beta\omega_{ij,k}v^k.
\end{align*}
On one hand we have
\begin{align*}
(d\omega)_{ijk}=\omega_{ij,k}+\omega_{ki,j}+\omega_{jk,i},
\end{align*}
\begin{align*}
(\iota_vd\omega)_{ij}=v^k(\omega_{ij,k}+\omega_{ki,j}+\omega_{jk,i}),
\end{align*}
and finally
\begin{align*}
(f^*\iota_vd\omega)_{\alpha\beta}=f^i_\alpha f^j_\beta v^k(\omega_{ij,k}+\omega_{ki,j}+\omega_{jk,i}).
\end{align*}
While on the other hand,
\begin{align*}
(\iota_v\omega)_i=v^j\omega_{ji},
\end{align*}
\begin{align*}
(f^*\iota_v\omega)_\alpha=f^i_\alpha v^j\omega_{ji},
\end{align*}
and
\begin{align*}
(df^*\iota_v\omega)_{\alpha\beta}=f^j_\beta v^i_\alpha\omega_{ij}-f^j_\beta v^k \omega_{jk,i}f^i_\alpha+f^i_\alpha v^j_\beta\omega_{ij}-f^i_\alpha v^k\omega_{ki,j}f^j_\beta.
\end{align*}
Adding these gives the result.
\end{proof}
Now, given a smooth of map $f:M\to N$ between almost Hermitian manifolds, the anti-holomorphic energy $E_+$ and holomorphic energy $E_-$ of $f$ decompose
\begin{align*}
E_\pm(f)=\frac{1}{4}\int_M|Tf\pm J_NTfJ_M|^2dV_g=E(f)\pm K(f)
\end{align*}
into a sum of the standard energy $E(f)=\frac{1}{2}\int_M|Tf|^2dV_g$ and an additional term
\begin{align*}
K(f)=-\frac{1}{2}\int_M\langle \omega_M,f^*\omega_N\rangle dV_g.
\end{align*}
Some obvious relations between these functionals are:
\begin{align*}
E&=\frac{1}{2}(E_++E_-)\\
K&=\frac{1}{2}(E_+-E_-).
\end{align*}
These belong to a family $E_a$ of energies which will be considered in a later section. These are given by
\begin{align*}
E_a=aE_++(1-a)E=E+aK.
\end{align*}
We next give the first variation of $K$. Remarkably, its Euler-Lagrange equation depends only on first derivatives of $f$.
\begin{prop}
Let $K=K(f)=-\frac{1}{2}\int_M\langle\omega_M,f^*\omega_N\rangle dV_g$ be the difference $E_+-E$ between the Dirichlet energy $E(f)=\frac{1}{2}\int_M|Tf|^2dV_g$ and the anti-holomorphic energy $E_+(f)=\frac{1}{4}\int_M|Tf+J_NTfJ_M|^2dV_g$. Let $v=\partial_t f$ be a variation of $f$ with compact support. Then
\begin{align*}
\frac{d}{dt}K(f) = -\int_M\langle v, A\rangle dV_g
\end{align*}
where $A$ is given by
\begin{align*}
2A = \langle d^*\omega_M,f^\times\omega_N\rangle^\sharp+\langle\omega_M,f^\times d\omega_N\rangle^\sharp,
\end{align*}
$f^\times\omega$ denotes the pullback of a form $\omega$ on all indices except for the first, and $\sharp:T^*\to T$ is the musical isomorphism given by the metric.
\end{prop}
\begin{proof}
Given the generalized Cartan formula of the previous proposition we differentiate under the integral to obtain:
\begin{align*}
\frac{d}{dt} K(f)& = -\frac{1}{2}\int_M \langle \omega_M,df^*\iota_v\omega_N+f^*\iota_vd\omega_N\rangle dV_g\\
& = -\frac{1}{2}\int_M \langle d^*\omega_M,f^*\iota_v\omega_N\rangle+\langle\omega_M,f^*\iota_vd\omega_N\rangle dV_g.
\end{align*}
\end{proof}
As a corollary we obtain the aforementioned result of Lichnerowicz in \cite{LN} for harmonic maps between almost K\"ahler manifolds.
\begin{cor}(Lichnerowicz)
If $d\omega_N=0$ and $d^*\omega_M=0$ then $K$ is a smooth homotopy invariant of $f$. Therefore, under these assumptions, the critical points of $E_+$ coincide with the critical points of $E$, i.e. harmonic maps. In particular, a holomorphic map between closed almost K\"ahler manifolds is harmonic and minimizes the energy in its homotopy class.
\end{cor}
We are also equipped to give a proof of Proposition 1.2.
\begin{proof}
Noting that $E^+=E+K$, we combine the first variation of $E$ with the first variation of $K$ to get the result.
\end{proof}
We next compute the second variation of $E_+$ and prove Proposition 1.3 and Corollary 1.4
\begin{proof}[Proof of Proposition 1.3]
From the first variation we have
\begin{align*}
\frac{\partial}{\partial t}E_+=-\int_M\langle v,\tau_+\rangle dV_g.
\end{align*}
We then compute
\begin{align*}
\N_t\tau_+=\N_t\tau+\N_tA& =\N_ttr_g\N Tf+\N_t A\\& = tr_g\N\N_tTf+tr_gR(v,Tf)Tf+\N_t A\\& =\Delta v+tr_gR(v,Tf)Tf+\N_tA,
\end{align*}
therefore
\begin{align*}
\frac{\partial^2}{\partial t^2}|_{t=0}E_+=-\int_M\langle v,\Delta v+tr_gR(v,Tf)Tf+\N_tA\rangle dV_g.
\end{align*}
Finally,
\begin{align*}
\N_t A&=\frac{1}{2}\{ \langle d^*\omega_M,\omega_N(\cdot,\N v)\rangle^\sharp\\& +\langle d^*\omega_M,(\N_v\omega_N)(\cdot,Tf)\rangle^\sharp+\langle\omega_M,(d\omega_N)(\cdot,\N v,Tf)\rangle^\sharp\\& +\langle\omega_M,(d\omega_N)(\cdot,Tf,\N v)\rangle^\sharp\\& +\langle\omega_M,(\N_v d\omega_N)(\cdot,Tf,Tf)\rangle^\sharp\}
\end{align*} accounts for the remaining terms.
\end{proof}

We clarify what is meant by the various inner products in the $\N_t A$ term of the previous proposition by expressing some of them in local coordinates. For example,
\begin{align*}
\langle\omega_M,(\N_vd\omega_N)(\cdot,Tf,Tf)\rangle_i=\omega_{\alpha\beta}(\N_vd\omega_N)_{ijk}f^j_\gamma f^k_\delta g^{\alpha\gamma}g^{\beta\delta}
\end{align*}
and
\begin{align*}
\langle \omega_M,(d\omega_N)(\cdot,Tf,\N v)\rangle_i=\omega_{\alpha\beta}f^j_\gamma f^l_\delta (d\omega)_{ijk}g^{\alpha\gamma}g^{\beta\delta}(\N_l v)^k.
\end{align*}

\begin{proof}[Proof of Corollary 1.4]
Consider maps $f:M\to N$. Note that $E_+(f)\geq 0$ and $E_+(f)=0$ if, and only if, $f$ is (pseudo) holomorphic. Therefore holomorphic maps are stable critical points of $E_+$ as they are global minimizers of this functional. 

Now, the identity map $id:M\to M$ is always holomorphic and is therefore a stable $E_+$ critical point, hence the operator $L$ is non-negative. If $v$ is a holomorphic vector field then it generates a one parameter family of holomorphic maps beginning at the identity and so $Lv=0$.
\end{proof}

\begin{proof}[Proof of Corollary 1.5]
Let $f:M\to \arr$ satisfying $\Delta f=-\lambda f$ be an eigenfunction. Recall the Bochner formula $\Delta\N f=\N\Delta f+Rc(\N f)$. If $(M,g,J)$ is almost K\"ahler then by the previous corollary the operator $L=-\Delta-Rc$ on smooth vector fields is non-negative. We then compute
\begin{align*}
L\N f=-\Delta\N f-Rc(\N f)=-\N\Delta f-2Rc(\N f)=\lambda\N f-2Rc(\N f).
\end{align*}
Thus, since $Rc\geq \alpha$, we have
\begin{align*}
0\leq\int_M\langle L\N f,\N f\rangle dV\leq(\lambda-2\alpha)\int_M |\N f|^2dV.
\end{align*}
\end{proof}

\section{The Negative Gradient Flow}
\subsection{The energy $E_+$}
In this section we will study the negative gradient flow corresponding to $E_+$. Consider the initial value problem
\begin{align*}
\frac{\partial}{\partial t}f=\tau_+(f)\\
f|_{t=0}=f_0.
\end{align*}
A solution to this problem is said to solve the $\delbar$-harmonic map heat flow with initial condition $f_0$. We've seen that $\tau_+=\tau+A$, while $A$ consists of two terms, one linear and one quadratic in $Tf$. Therefore the linearized operator of $\tau_+$ has the same principal symbol as that of $\tau$. Applying a dose of semi-linear parabolic existence and regularity theory we therefore obtain
\begin{prop}
Let $M$ be a closed, smooth, almost Hermitian manifold and let $f_0:M\to N$ be a smooth map from $M$ to the smooth, almost Hermitian manifold $N$. Then there is a $T>0$ such that the initial value problem
\begin{align*}
\frac{\partial}{\partial t}f=\tau_+(f)\\
f|_{t=0}=f_0
\end{align*}
has a unique smooth solution on $[0,T)\times M$.
\end{prop}

We next obtain some basic apriori estimates for the $\delbar$-harmonic map heat flow which will give sufficient conditions to conclude long time existence.
\begin{prop}\label{estimate}
Let $f$ be a solution to the $\tau_+$-flow. Then
\begin{align*}
\N_t\tau_+=\Delta\tau_++tr_gR(\tau_+,Tf)Tf+\N_tA,
\end{align*}
and
\begin{align*}
\N_tTf=\Delta Tf-Q+\N A.
\end{align*}
Therefore
\begin{align*}
(\del_t-\Delta)\frac{1}{2}|\tau_+|^2=-|\N\tau_+|^2+\langle tr_gR^N(\tau_+,Tf)Tf,\tau_+\rangle+\langle\N_tA,\tau_+\rangle,
\end{align*}
and
\begin{align*}
(\del_t-\Delta)\frac{1}{2}|Tf|^2=-|\N Tf|^2-\langle Q,Tf\rangle+\langle\N A,Tf\rangle,
\end{align*}
\begin{proof}
We compute
\begin{align*}
\N_t\tau_+&=\N_t\tau+\N_tA\\
&=tr_g\N_t\N Tf+\N_tA\\
&=tr_g\N^2\tau_+ +tr_g R^N(\tau_+,Tf)Tf+\N_t A\\
&=\Delta\tau_+ +tr_g R^N(\tau_+,Tf)Tf+\N_t A,
\end{align*}
using the commutator formula $\N_t\N-\N\N_t=R^N(\dot{f},Tf)$. This proves the first claim. The second claim follows from a similar calculation
\begin{align*}
\N_t Tf=\N\tau_+&=\N\tau+\N A\\
&=\Delta Tf-Q+\N A,
\end{align*}
where 
\begin{equation*}
Q(\cdot)=-R^N(Tf\cdot,Tf_\alpha)Tf_\alpha+Tf(Rc^M\cdot)
\end{equation*}
and we have used the Bochner formula
\begin{equation*}
\Delta Tf=\N\tau+Q.
\end{equation*}
The final claims follow from
\begin{align*}
\partial_t\frac{1}{2}|\tau_+|^2&=\langle \N_t\tau_+,\tau_+\rangle\\
&=\langle\Delta\tau_+ +tr_g R^N(\tau_+,Tf)Tf+\N_t A,\tau_+\rangle\\
&=\Delta \frac{1}{2}|\tau_+|^2-|\N \tau_+|^2+\langle tr_g  R^N(\tau_+,Tf)Tf+\N_t A,\tau_+\rangle,
\end{align*}
with a similar calculation for $\partial_t\frac{1}{2}|Tf|^2$.
\end{proof}
\end{prop}
\begin{prop}
If the solution to the $\delbar$-harmonic heat flow between two compact, almost Hermitian manifolds exists on a maximal time interval $[0,T)$, with $T<\infty$, then $\limsup_{t\uparrow T}|Tf|_\infty=\infty$.
\begin{proof}
This is a standard obstruction in semi-linear parabolic systems. If there is some positive constant $C$ such that $|Tf|<C$ on $[0,T)$ then we conclude convergence to a smooth map at time $T$. Using this map as the new condition for the flow, we extend the solution smoothly past $T$ and contradict maximality.
\end{proof}
\end{prop}
We can now prove a long time existence result for $\delbar$-harmonic map heat flow, a corollary of which is Theorem 1.5.
\begin{prop}\label{longtime}
Suppose that $M$ is compact and the target $N$ is almost Kahler with sectional curvature bounded above by a negative constant and with bounded Nijenhaus tensor. Then the $\delbar$-harmonic map heat flow beginning at any smooth $f_0:M\to N$ exists smoothly for all time.
\begin{proof}
Let $f$ be the corresponding solution to the flow. It amounts to proving that $|Tf|$ is bounded on any interval of the form $[0,T)$ where $T<\infty$. Recall from \ref{estimate} that
\begin{align*}
(\del_t-\Delta)\frac{1}{2}|Tf|^2=-|\N Tf|^2-\langle Q,Tf\rangle+\langle\N A,Tf\rangle.
\end{align*}
Since the target is almost K\"ahler we have $2A^\flat=\langle d^*\omega_M,f^\times\omega_N\rangle$. We then compute $\langle\N A,Tf\rangle$, which has the form
\begin{align*}
2\langle\N A,Tf\rangle=\N d^*\omega_M*Tf\wedge Tf+d^*\omega_M*\N Tf\wedge Tf+d^*\omega_M*\N^h\omega_N*(Tf\wedge Tf)*Tf,
\end{align*}
where we have used schematic notation. To clarify, if $A$ and $B$ are some tensors then $A*B$ denotes a tensor constructed from taking contractions, either with the metrics $g$, $h$ or the forms $\omega_M$, $\omega_N$, of $A\otimes B$ and $B\otimes A$. We have also used the notation that, if $A$ and $B$ are in $T^*M\otimes E$ for some vector bundle $E$ then $A\wedge B(X,Y)=A(X)\wedge B(Y)$. In particular, since $M$ is compact and $\N^h\omega_N$ is bounded, there is some constant $C_1$ independent of $f$ such that
\begin{align*}
\langle\N A,Tf\rangle\leq C_1(|Tf|^2+|\N Tf||Tf|+|Tf \wedge Tf||Tf|).
\end{align*}

Now, if the sectional curvatures of $N$ are bounded above by $K$ and the Ricci curvature of $M$ is bounded below by $R$, then the curvature term $\langle Q,Tf\rangle$ satisfies
\begin{align*}
-\langle Q,Tf\rangle\leq K|Tf\wedge Tf|^2-R|Tf|^2.
\end{align*}
An application of Young's inequality $ab\leq \frac{1}{2}(\epsilon a^2+\frac{1}{\epsilon}b^2)$ then gives
\begin{align*}
-|\N Tf|^2-\langle Q,Tf\rangle+\langle\N A,Tf\rangle&\leq C_1(1+\frac{1}{2\epsilon_1}+\frac{1}{2\epsilon_2})|Tf|^2-R|Tf|^2\\&+(\frac{C_1\epsilon_1}{2}-1)|\N Tf|^2+(\frac{C_1\epsilon_2}{2}+K)|Tf\wedge Tf|^2.
\end{align*}
Because $K$ is negative, by choosing $\epsilon_1$ and $\epsilon_2$ small enough the last two terms of this expression are negative. Therefore there is some constant $C$ such that
\begin{align*}
\left(\partial_t-\Delta\right)|Tf|^2\leq C|Tf|^2,
\end{align*}
and so $|Tf|$ is bounded on any finite length time interval by the maximum principle.
\end{proof}
\end{prop}
We note that if the complex structure of $N$ is integrable, so that $N$ is genuinely K\"ahler, then there is no $\N^h\omega_N$ term in the above estimates. This would mean we can weaken the negative sectional curvature assumption to just non-positive curvature. With a more careful analysis several of the assumptions in this theorem can be weakened; certainly regularity in $f_0$ or compactness/boundedness assumptions can be weakened. We make the following long time existence conjecture, independent of any integrability assumptions of any kind on source or target.
\begin{conj}
If $M$ is compact, the sectional curvature of $N$ is non-positive, and $\N J_N$ is bounded, then the $\delbar$-harmonic map heat flow beginning with any smooth $f_0:M\to N$ exists smoothly for all time.
\end{conj}
A quick look at the form of $\langle\N A,Tf\rangle$ should indicate to the reader why obtaining such a result using a basic maximum principle type argument as above is a little delicate. Without the almost K\"ahler assumption on $N$ there are three additional terms which have the schematic form
\begin{align*}
\N\omega_M*d\omega_N*Tf\wedge Tf\wedge Tf+d\omega_N*\N Tf\wedge Tf\wedge Tf+\N^h d\omega_N*Tf\wedge Tf\wedge Tf*Tf.
\end{align*}
There appears to at least be some conditions on the curvature of $N$ in relation to the covariant derivatives of $J_N$ to conclude long time existence, but we will not investigate this any further in this paper.
\subsection{The Perturbed Energies}
In this section we will consider the perturbed energies $E_a=E+aK$. Our first lemma is obvious from the previous sections.
\begin{lemma}
The Euler-Lagrange equation of $E_a$ is $\tau_a=\tau+aA$.
\end{lemma}
We also readily have a long time existence result for the flow associated to $E_a$, as well as convergence and the proof of \ref{aconv} if $|a|<1$.
\begin{proof}
Let $f$ be the solution to the flow. The long time existence of the flow was established in \ref{longtime}, where we note that the change from $E_+$ to $E_a$ is a trivial modification of the proof. Specifically, along a solution to the flow there is a constant $C_1>0$ depending on $a$ and background data such that
\begin{align*}
(\partial_t-\Delta)|Tf|^2\leq C_1|Tf|^2,
\end{align*}
from which we conclude long time existence. We then must prove convergence at infinite time, which we do so in essentially the same way as for the Eells-Sampson result.

Since $(1-|a|)E\leq E_a$, we have a uniform energy bound $E<C$ along a solution to the flow when $|a|<1$. Recall Moser's Harnack inequality for subsolutions to the heat equation \cite{MOSER}: if $g$ is non-negative and there is some positive constant $C$ so that $(\partial_t-\Delta)g\leq Cg$ in a parabolic cylinder $P_R(x_0,t_0)=\{(x,t)|d(x,x_0)\leq R, t_0-R^2\leq t\leq t_0\}$ centered at $(x_0,t_0)$, then there exists another constant $C'>0$ such that
\begin{align*}
g(x_0,t_0)\leq C'R^{-(n+2)}\int_{P_R(x_0,t_0)}gdV.
\end{align*}
We then apply this inequality to $g=|Tf|^2$, giving, for some constant $C$ depending only on background data and $a$,
\begin{align*}
|Tf|^2(x,t)&\leq CR^{-(n+2)}\int_{P_R(x_0,t_0)}|Tf|^2 \\
&\leq CR^{-(n+2)}\int_{t-R^2}^tE(f(s))ds\\
&\leq \frac{C}{1-|a|}R^{-(n+2)}\int_{t-R^2}^tE_a(f(s))ds\\
&\leq \frac{C}{1-|a|}R^{-n}E_a(f_0).
\end{align*}
We therefore have a uniform bound on $|Tf|$ on all of $[0,\infty)$. By the higher regularity theory for second-order parabolic equations we conclude the existence of constants $C(f_0,M,N,k,a)$ such that $\sup_{M\times[0,\infty)}|\N^k Tf|\leq C(f_0,M,N,k,a)$.
Note also that there is some constant $C$ such that
\begin{align*}
(\partial_t-\Delta)|\tau_a|^2\leq C|\tau_a|^2.
\end{align*}
Again by Moser's Harnack inequality we conclude
\begin{align*}
|\tau_a|^2(x,t)&\leq CR^{-(n+2)}\int_{t-R^2}^t\int_M|\tau_a|^2dVds\\
&= CR^{-(n+2)}\int_{t-R^2}^t-\partial_sE_ads\\
&= CR^{-(n+2)}(E_a(t-R^2)-E_a(t)),
\end{align*}
and so $|\tau_a|^2\to 0$ as $t\to\infty$. Taking any sequence of $t_i$ going to $\infty$, by the previous estimates we can extract a subsequence such that $f(t_i)$ converges in any $C^k$ norm to some $f_\infty$ satisfying $\tau_a(f)=0$.
\end{proof}
\subsection{Uniform Energy Bounds}
The previous result, while satisfactory from an analytic point of view, is far from proving convergence for the $\delbar$-harmonic map heat flow for general non-positively curved targets, and it is unfortunate that we had to consider the coercive energies $E_a$ instead of the energy $E_+$, which is the focus of the paper. The standard energy $E$ of a solution to the $\delbar$-harmonic map heat flow may not be bounded along a solution to the flow, but such a uniform bound is necessary for many known parabolic and elliptic estimates to apply. We know of no examples where such a uniform bound fails to hold, but we cannot rule it out in general. If the target is almost K\"ahler we at least have such a bound on any finite length time interval.
\begin{prop}
If $N$ is almost K\"ahler then a solution to $\delbar$-heat flow has bounded energy on time intervals of finite length.
\begin{proof}
Along a $C^2$ solution to the flow
\begin{align*}
\frac{d}{dt}E=-\int |\tau|^2-\int_M\langle \tau,A\rangle
\end{align*}
Since $N$ is almost K\"ahler, we have that $A$ is some tensor which is linear in $Tf$ and contracted with only $d^*\omega_M$ and $\omega_N$. Therefore there are some positive constants $C_1$ and $C_2$ such that $-\langle \tau,A\rangle\leq C_1|\tau||Tf|\leq\frac{1}{2}(|\tau|^2+C_2|Tf|^2)$. Therefore
\begin{align*}
\frac{d}{dt}E\leq C_3E
\end{align*}
for some constant $C_3$. Therefore
$E(t)\leq E(0)e^{C_3 t}$.
\end{proof}
\end{prop}
If our manifolds are uniformly equivalent to balanced and almost K\"ahler manifolds, then remarkably an energy bound does hold.
\begin{prop}
Suppose that $(M,g,J_M)$ is compact and uniformly equivalent to a $J_M$-compatible balanced metric $g_0$, so that $d^*\omega_{g_0}=0$. Suppose also $(N,h,J_N)$ is uniformly equivalent to a $J_N$-compatible almost K\"ahler metric $h_0$, non necessarily complete. Suppose further that $f_t$ is a smooth one-parameter family of maps such that $E_+(f_t)$, computed with respect to $g$ and $h$, is uniformly bounded. Then there is a uniform bound on $E(f_t)$ computed with respect to $g$ and $h$.
\begin{proof}
Let $E^{gh}(f_t)$ and $E_+^{gh}(f_t)$ denote the energy and pseudoholomorphic energy of $f_t$ computed with respect to the metrics $g$ and $h$, and let $K^{gh}(f_t)$ be the difference between these. Note that $K^{g_0h_0}(f_0)$ is a smooth homotopy invariant of $f_0$, and so $K^{g_0h_0}(f_t)$ is constant. By the uniform equivalence of the metrics, we have
\begin{align*}
E^{gh}(f_t)\leq CE^{g_0h_0}(f_t)&=C(E_+^{g_0h_0}(f_t)-K^{g_0h_0}(f_t))\\
&=C(E_+^{g_0h_0}(f_t)-K^{g_0h_0}(f_0))
\end{align*}
for some constant $C>0$ witnessing the equivalence of the metrics. But again by the uniform equivalence of the metrics
\begin{align*}
E_+^{g_0h_0}(f_t)\leq CE_+^{gh}(f_t).
\end{align*}
and so $E^{gh}(f_t)$ is uniformly bounded.
\end{proof}
\end{prop}
The following corollary is immediate from the previous proposition, and gives a rough picture of what is occurring when we fail to have a uniform energy bound along a solution to the $\delbar$-harmonic map heat flow.
\begin{cor}
Let $f:M\times[0,T)\to N$, $0<T\leq\infty$ be a smooth solution to the $\delbar$-harmonic map heat flow between a compact, balanced, almost Hermitian manifold $M$ and a (not necessarily complete) almost Hermitian manifold $N$. If there is not a uniform energy bound $E(f_t)\leq C$ for all $t\in[0,T)$, then for each $t_0\in[0,T)$ no neighborhood of $f(M\times[t_0,T))$ in $N$ can admit a uniformly equivalent K\"ahler metric. 
\end{cor}

\section{$\delbar$-harmonic Map Flow and $\delbar$-harmonic maps of Riemannian Surfaces}
\subsection{An example of the Flow into a Non-K\"ahler Surface}
Pick some $\alpha>1$ and let $N=S^3\times S^1=\cee^2\setminus\{0\}/\mathbb{Z}$ be the Hopf surface generated by the $\mathbb{Z}$ action given by $(z,w)\mapsto(\alpha z,\alpha w)$ on $\cee^2\setminus\{0\}$. $N$ is a compact complex manifold which cannot admit any K\"ahler metrics for Hodge theoretic reasons. If $\rho$ denotes the distance to $0$ in $\cee^2$, a Hermitian metric $h$ on $N$ has corresponding two form $\omega_N(\cdot,\cdot)=h(J_N\cdot,\cdot)$ given by
\[\omega_N=\frac{1}{2\rho^2}\ii\partial\bar{\partial}\rho^2.\]
Notice that this is indeed a metric on $N$, as it is invariant under scalar multiplication and unitary transformations, and is moreover locally conformal to the standard Euclidean metric on $\arr^4$.
A similar construction applied to $\cee^*$ gives a torus $M=T^2=\cee^*/\mathbb{Z}$, and K\"ahler metric $g$ with corresponding K\"ahler form
\[\omega_M=\frac{1}{2\rho^2}\ii \partial\bar{\partial}\rho^2,\]
which is likewise invariant, in fact it is a flat metric on $M$.

Pick some orthonormal basis $e_i$ of $\arr^4$ so that the standard complex structure on $\cee^2\sim\arr^4$ takes the form $Je_1=e_2$ and $Je_3=e_4$. Take coordinates $y^i$ on $\arr^4$ induced by this basis. Do a similar construction for $\arr^2\sim\cee$ and call the corresponding coordinates $x^i$. For any pair $(u,v)$ of orthonormal vectors in $\arr^4$, consider the $\arr$-linear map $f:\arr^2\to \arr^4$ given by
\[f(x^1,x^2)=x^1u+x^2v.\]
Notice that $f$ is an orthogonal embedding since $u$ and $v$ are orthonormal. Linearity implies that $f$ is equivariant with respect to the $\mathbb{Z}$ action and therefore descends to a map $f:M\to N$ of the torus into the Hopf surface. By the orthogonality, it is evident that $f^*h=g$ and, moreover, $f$ is totally geodesic. This is most easily seen by noting that $h$ can be viewed as a bi-invariant metric with respect to some Lie group structure on $N$ for which $f:M\to N$ is the inclusion of a torus subgroup.

These $f$ are therefore a family of harmonic maps $f:M\to N$ parameterized by orthonormal pairs in $\arr^4$. We can compute $K(f)$ for these maps directly. First, note that since $\dim_\cee M=1$,
\[K(f)=-\frac{1}{2}\int_M\langle\omega_M,f^*\omega_N\rangle\omega_M=-\int_Mf^*\omega_N.\]
In the coordinates $y^i$
\[\omega_N=\frac{1}{\rho^2}(dy^1\wedge dy^2+dy^3\wedge dy^4),\]
and so
\begin{align*}
f^*\omega_N&=\frac{1}{\rho^2}(u^1v^2-u^2v^1+u^3v^4-u^4v^3)(dx^1\wedge dx^2)\\&=(u^1v^2-u^2v^1+u^3v^4-u^4v^3)\omega_M,
\end{align*}
therefore
\begin{align*}
K(f)=-(u^1v^2-u^2v^1+u^3v^4-u^4v^3)V,
\end{align*}
where $V=2\pi\log\alpha$ is the volume of $T^2$ with respect to $g$. In particular, it is clear from this example that $K$ is not a homotopy invariant. In fact, the homotopy invariance of $K$ is exactly what allows the pseudoholomorphic energy to distinguish holomorphic maps from harmonic maps.
\begin{prop}
Let $\mathcal{F}$ denote the family of all such $f:T^2\to S^3\times S^1$ constructed as above. Then this family is preserved by the $\delbar$-harmonic map heat flow and, for any $f_0\in\mathcal{F}$, the flow exists for all time and converges subsequentially to a holomorphic or anti-holomorphic map $f_\infty:T^2\to S^3\times S^1$.
\begin{proof}
Note that $\langle A(f),Tf\rangle=0$. This is because this depends on $d\omega_N(Tf,Tf,Tf)$ but $Tf$ only has rank 2. Therefore, since each $f\in\mathcal{F}$ is harmonic, we have that $\tau_+=\tau+A=A$ is perpendicular to the image of $f$ at each point. This, together with the fact that $f$ is linear and so $A$ is linear in the coordinates $x^i$, implies that the evolution equation $\partial_tf=A$ preserves the family. We can then view the flow in this family as given by a smooth vector field on the space of oriented orthonormal 2-frames in $\arr^4$, establishing long time existence. Convergence to a $\delbar$-harmonic map is then immediate from monotonicity of $E_+$ and compactness of this family.

To see that the limiting map must be holomorphic or anti-holomorphic, note that the energy $E_+$ is invariant $E_+(U\circ f)=E_+(f)$ under the unitary group of $\cee^2$. Since these act transitively on the unit vectors, we can assume the limiting map has $u=e_1$. But then a direct computation shows that, when $u=e_1$, $A(f)=0$ if, and only if, $v=\pm e_2$. Thus the limiting map is holomorphic or anti-holomorphic.
\end{proof}
\end{prop}
This relatively simple example demonstrates that the flow can distinguish a holomorphic map from a harmonic map in non-K\"ahler settings. In light of Corollary 3.9, if there is a singularity in the flow and a uniform energy bound does not hold, then the image of the flow near the singularity must be quite wild in $S^3\times S^1$; any proper compact subset of $S^3\times S^1$ admits a K\"ahler metric which, due to compactness, is uniformly equivalent to $h$. Therefore the lack of a uniform energy bound would mean the solution leaves every proper compact subset of the Hopf surface.
\subsection{Bubbling}
In this section we will consider the $\delbar$-harmonic map heat flow for maps $f:\Sigma\to N$ between a compact Riemann surface $\Sigma$ and a compact, almost Hermitian manifold $N$. In this setting the pseudoholomorphic energy $E_+$ is conformally invariant, as is readily seen from the fact that we now have
\begin{align*}
E_+(f)=\int_\Sigma\frac{1}{2}|Tf|^2dV-\int_{\Sigma}f^*\omega_N.
\end{align*}
The most important observation we can make is that in this form the functional is exactly amenable to the fantastic result of Riviere \cite{RIVCONS} on the regularity of two variable conformally invariant elliptic systems. Specifically, we will quote the following theorem.
\begin{thrm}[Theorem 1.1 of \cite{RIVCONS}]
Let $B$ be a ball in $\arr^2$ and let $u\in W^{1,2}(B,\arr^n)$ be a weak solution to the system
\begin{align*}
\Delta u^i=\Omega^i_j(\N u^j),
\end{align*}
where $\Omega\in L^2(B,\mathfrak{so}(n)\otimes\wedge^1\arr^2)$. Then $u$ is locally H\"older continuous in $B$.
\end{thrm}
We derive a number of corollaries in applying this result to $\delbar$-harmonic maps of surfaces. Analogues of these are well known in the theory of harmonic maps.
\begin{cor}
Let $N\subseteq\arr^n$ be a smooth, compact, almost Hermitian manifold. Let $u\in W^{1,2}(B,N)$ be a weakly $\delbar$-harmonic map. Then $u$ is smooth. In particular, if $u:B/\{0\}\to N$ is smooth and $\delbar$-harmonic with finite energy, then $u$ is smooth in $B$.
\begin{proof}
As observed in Theorem 1.2 of \cite{RIVCONS}, any conformally invariant quadratic energy functional in two-dimensions has Euler-Lagrange equation in the form for which the previous theorem applies. Specifically, any functional of the form
\begin{align*}
\frac{1}{2}\int_{\Sigma}|Tu|^2dV+\int_{\Sigma}u^*\omega,
\end{align*}
where $\omega$ is any $C^1$ section of $\wedge^2T^*N$ has Euler-Lagrange equation in the form required by the theorem. The $E_+$ energy is exactly of this form, so any weakly $\delbar$-harmonic map $u:B\to N$ for which $E(u)<\infty$ is H\"older continuous. Smoothness of $u$ then follows from the smoothness of $N$, the H\"older continuity of $u$, and higher regularity theory of elliptic systems.
\end{proof}
\end{cor}
Note that in the previous corollary the assumption of finite energy is essential, as the map $z\mapsto z^{-1}$ is clearly $\delbar$-harmonic with $E_+$ finite but is not smooth in a ball centered at the origin.
\begin{cor}
If $u:\arr^2\to N$ is $\delbar$-harmonic and has finite energy, then $u$ extends to a smooth $\delbar$-harmonic map $\tilde{u}:S^2\to N$.
\begin{proof}
Consider the map given by composing $u$ with stereographic projection from $S^2$ to $\arr^2$. The previous proposition then implies that there is a smooth extension of this map to the point at infinity.
\end{proof}
\end{cor}
\begin{cor}
Suppose that a solution $u$ to the $\delbar$-harmonic map heat flow from a compact Riemann surface $\Sigma$ to a compact, almost Hermitian manifold $N$ exists on a maximal time interval $[0,T)$, where $T<\infty$, and there is a uniform energy bound on the solution. Then there exists a point $p\in\Sigma$, a sequence of times $t_i\nearrow T$, and a sequence of $r_i\searrow 0$ such that the family of maps $u_i(x)=u(exp_p(r_ix),t_i)$ converges to a limiting map $u_\infty:\arr^2\to N$ in $H^{2,2}_{loc}$ to a non-constant, smooth harmonic map with finite energy.
\begin{proof}
As shown in Proposition 3.3 we know there is a sequence of times $t_i\nearrow T$ and points $p_i\to p\in \Sigma$ such that $\lim_{i\to\infty}|Tf|(p_i,t_i)=\infty$ and $|Tf|(p_i,t_i)=\sup_{p\in\Sigma,t\leq t_i}|Tf|(p,t)$. Let $r_i^{-1}=|Tf|(p_i,t_i)$ and consider a geodesic ball centered at $p$ of some small radius $\rho$. For $x\in B_{r_i^{-1}\rho}(0)\subset T\Sigma_p$ let
\begin{align*}
u_i(x,t)=u(exp_p(xr_i),t_0+r_i^2t).
\end{align*}
Note that $u_i$ solves the $\delbar$-harmonic map heat flow with respect to the metric $g_i=exp_p(r_i\cdot)^*g$. Since this metric is converging in $C^2_{loc}$ to the flat metric on $\arr^2$ we can extract a subsequence $u_i\to u_\infty$ converging locally in $C^2(\arr^2\times(-\infty,0],N)$ where $u_\infty:\arr^2\times(-\infty,0]\to N$ is a non-trivial solution to the $\delbar$-harmonic map heat flow with finite energy and constant $E_+$-energy, in particular it is a $\delbar$-harmonic map with finite energy.
\end{proof}
\end{cor}
\begin{cor}
With the assumptions of the previous corollary, there must exist a $\delbar$-harmonic sphere in $N$. In particular, if $N$ does not admit a non-trivial $\delbar$-harmonic $S^2$, then any solution to the $\delbar$-harmonic map heat flow with a uniform energy bound from any compact Riemann surface $\Sigma$ to $N$ exists smoothly for all time.
\begin{proof}
By the previous corollary, if a finite time singularity occurs then there is a non-trivial $\delbar$-harmonic map $u_\infty:\arr^2\to N$ with finite energy. This $u$ then extends by Corollary 4.3 to a $\delbar$-harmonic map of $S^2$ into $N$.
\end{proof}
\end{cor}

\section{Concluding Remarks}
We have cataloged a number of results on $\delbar$-harmonic maps which are the analogues of results known for harmonic maps. Our eventual goal will be to use $\delbar$-harmonic maps to study almost Hermitian manifolds, specifically non-K\"ahler complex manifolds. We end the paper with a list of important questions/problems which would serve as a starting point for future work.
\begin{itemize}
\item Is there always a uniform energy bound up to the first singular time $T$ for a solution to the flow? As we have seen, the lack of such a bound for a Riemann surface implies the spacetime image $f(\Sigma\times [t_0,T))$, for any $t_0$, has no neighborhood which admits an equivalent almost K\"ahler metric. Perhaps the simplest situation where a uniform energy bound would fail to hold along the flow would be for maps $f:\Sigma\to S$ from a compact Riemann surface into an Inoue surface. These are compact complex surfaces where there are no holomorphic maps $f:\Sigma\to S$ for any compact Riemann surface $\Sigma$, see \cite{INOUE} for their construction. On the other hand, given Proposition 3.8 one might expect a uniform energy bound holds whenever the source is balanced and the target is uniformly locally almost K\"ahler.
\item Under what conditions is a $\delbar$-harmonic map pseudo-holomorphic? For K\"ahler manifolds there is the celebrated complex analyticity result of Siu \cite{SIU} for maps of strongly negatively curved K\"ahler manifolds. There are also known conditions (see Chapter 8, Section 3 of \cite{SY}, for example) for stable harmonic maps $f:S^2\to N$ into a K\"ahler manifold which are sufficient to conclude complex analyticity. Both results strongly rely on the K\"ahler assumption.
\item Recently Rupflin and Topping \cite{RT} have considered the Teichm\"uller Harmonic map flow for maps $u:\Sigma\to N$ from a Riemann surface $\Sigma$ into a Riemannian manifold $N$. Their flow is a coupling of the harmonic map heat flow to a flow of the complex structure on $\Sigma$. We are interested in playing a similar game for the $\delbar$-harmonic map heat flow. Specifically, for a given complex structure on $\Sigma$ there may not be any holomorphic maps $u:\Sigma\to N$ and so if our goal is to locate holomorphic maps within a homotopy class, one expects that a change in complex structure is necessary.
\item Given two almost Hermitian structures on $S^6$ or $\cee P^3$, what can be said about the $\delbar$-harmonic map heat flow beginning from the identity map of these spaces? A small but notable observation here is that the homotopy class of the identity map of $S^6$ never has a global minimizer of the harmonic map energy, but the identity map is always a globally minimizing $\delbar$-harmonic map.
\end{itemize}

\bibliographystyle{hamsplain}

\end{document}